\documentclass{amsart}
 \usepackage{amsthm,amsfonts,amsmath,amssymb,latexsym,epsfig,upref,eucal,ae}
\usepackage[all]{xy}

\usepackage{graphicx}
\usepackage{subfigure}
\usepackage{epic}
\usepackage{eepic}
\usepackage{setspace}
\usepackage{booktabs}
\usepackage{longtable}
\usepackage{pstricks}
\usepackage{url}

\usepackage{color}

\newtheorem{theorem}[subsection]{Theorem}
\newtheorem{lemma}[subsection]{Lemma}
\newtheorem{proposition}[subsection]{Proposition}
\newtheorem{corollary}[subsection]{Corollary}

\newenvironment{remark}{\medskip \refstepcounter{subsection}
\noindent  {\bf Remark \thetheorem}.\rm}{\,}
\newenvironment{definition}{\medskip \refstepcounter{subsection}
\noindent  {\bf Definition \thetheorem}.\rm}{\,}

\newtheorem{theointro}{Theorem}

\newtheorem{propintro}[theointro]{Proposition}

\newtheorem{rmkintro}[theointro]{Remark}

\def\Aut{\mathrm{Aut}}

\def\cF{\mathcal{F}}
\def\cO{\mathcal{O}}

\def\cH{\mathcal{H}}

\def\cB{\mathcal{B}}
\def\cL{\mathcal{L}}

\def\cX{\mathcal{X}}

\def\del{\partial}
\def\delb{\overline\partial}

\def\h{\mathfrak{h}}

\def\g{\mathfrak{g}}

\def\k{\mathfrak{k}}

\def\R{\mathbb{R}}

\def\N{\mathbb{N}}
\def\T{\mathbb{T}}
\def\P{\mathbb{P}}

\def\C{\mathbb{C}}

\def\uS{\underline{S}}

\def\om{\omega}
\def\ep{\varepsilon}

\def\>{\rangle}

\def\<{\langle}
\def\>{\rangle}

\begin{document}

\title[Lower bounds on the modified K-Energy and complex deformations]
{Lower bounds on the modified K-Energy and complex deformations}
\author[A. Clarke
]{Andrew Clarke}
\address{Instituto de Matem\'atica, Universidade Federal do Rio de Janeiro,
Av. Athos da Silveira Ramos 149,
Rio de Janeiro, RJ, 21941-909,
Brazil
.}
\email{andrew@im.ufrj.br}
\author[C. Tipler]{Carl Tipler}
\address{D\'epartement de math\'ematiques
Universit\'e du Qu\'ebec  Montr\'eal
Case postale 8888, succursale centre-ville
Montr\'eal (Qu\'ebec) H3C 3P8}
\email{carl.tipler@cirget.ca}

\date{\today}

\begin{abstract}
Let $(X,L)$ be a polarized K\"ahler manifold that admits an extremal metric in $c_1(L)$. We show that on a nearby polarized deformation $(X',L')$ that preserves the symmetry induced by the extremal vector field of $(X,L)$, the modified K-energy is bounded from below.
This generalizes a result of Chen, Sz\'ekelyhidi and Tosatti (\cite{ch2,sz10,to}) to extremal metrics. Our proof also extends a convexity inequality on the space of K\"ahler potentials due to X.X. Chen \cite{ch} to the extremal metric setup.
As an application, we compute explicit polarized $4$-points blow-ups of $\C\P^1\times\C\P^1$ that carry no extremal metric but with modified K-energy bounded from below.
\end{abstract}

\maketitle

\section{Introduction}
\label{secintro}

Let $(X,L)$ be a polarized K\"ahler manifold. The Yau-Tian-Donaldson conjecture relates the existence of a constant scalar curvature K\"ahler metric  (CSCK metric) with K\"ahler class $c_1(L)$ to the GIT stability of the pair $(X,L)$, (see \cite{yau,tian,Don02}). This conjecture is motivated by the ``standard picture" \cite{Don97}, and in this framework the CSCK metrics are the critical points of the Mabuchi functional, or K-energy, introduced by Mabuchi \cite{ma}.  Donaldson has shown that if $\Aut(X,L)$ is discrete, then a CSCK metric in $c_1(L)$ is the limit of balanced metrics, which implies the uniqueness of the CSCK metric in its K\"ahler class \cite{Don01}, and the minimization of the Mabuchi functional \cite{Don05} by this metric. Chen \cite{ch}, and later Chen and Sun \cite{cs} gave new proofs of the minimization property with no assumption on the automorphism group. An even simpler proof is due to Li \cite{li}.

\par
Constant scalar curvature K\"ahler metrics give examples of the extremal metrics of Calabi \cite{c1}. An extremal metric is a critical point of the functional that assigns to each K\"ahler metric in the K\"ahler class the square of the $L^2$ norm of its scalar curvature, with respect to the volume form that also comes from the metric. An important property of extremal metrics is that the connected component of the identity of the isometry group is a maximal compact connected subgroup of the reduced automorphism group of the manifold \cite{c2}. Studying extremal metrics requires us to work modulo such a maximal compact connected group $G^m\subset\Aut(X,L)$. For example, Szek\'elyhidi \cite{sz} gave a relative version of the above conjecture for extremal metrics. Extremal metrics can also be seen to be critical points for a \emph{modified} $K$-energy $E^{G^m}$, as introduced in \cite{gu,si,ct}. 
The uniqueness result of Donaldson using quantization has been generalized to extremal metrics by Mabuchi \cite{ma04}. In \cite{st}, it is shown, also by quantization, that in the polarized case extremal metrics are minima of the modified K-energy. Chen and Tian show with no polarization assumption that extremal metrics are unique in a K\"ahler class up to automorphisms and minimize the modified K-energy \cite{ct}.

\par
The aim of this paper is to study the lower boundedness property of the modified K-energy under complex deformation (see Definition~\ref{def:deformation}). Let $(X,L)$ be an extremal polarized K\"ahler manifold, i.e.  assume that $X$ carries an extremal K\"ahler metric in the class $c_1(L)$. Then the existence
 of an extremal K\"ahler metric on a nearby $(X',L')$ is subject to a finite dimensional stability condition, see \cite{sz10}, \cite{br} or \cite{rt}.
However, in the constant scalar curvature case, it follows from theorems of Sz\'ekelyhidi and Chen \cite{sz,ch2} that the K-energy remains bounded on nearby deformations. A simplified proof of the theorem of Chen was given by Tosatti \cite{to}.
   In the polarized case, by a theorem of Futaki and Mabuchi  \cite{fm}, an extremal metric admits an $S^1$-action by isometries induced by the extremal vector field. If one wants to smoothly deform an extremal metric along a complex deformation, a necessary condition is to deform the action corresponding to the extremal vector field along the fibers of the deformation. Our main result states that this condition is enough to ensure the lower boundedness of the modified K-energy on complex deformations of extremal polarized manifolds. 
 \par
To state our main theorem, we need to introduce some notation and terminology. For a polarized complex deformation $\cL \rightarrow \cX \stackrel{\pi}{\rightarrow} \mathcal{B}$ of some complex manifold $X=X_{t_0}=\pi^{-1}(t_0)$ for $t_0,t\in \mathcal{B}$, we denote by $J_t$ the almost-complex structure of the complex manifold $X_t$ and we denote by $V_t^{G^m_t}$ the extremal vector field of $(X_t,c_1(L_t))$ with respect to a maximal compact connected subgroup $G^m_t\subset\Aut(X_t,L_t)$. Given a group $G$, we will say that the polarized complex deformation is {\it $G$-invariant} if $G$ acts on the triple $(\cB,\cX,\cL)$, its action commutes with the maps $\cL \rightarrow \cX \rightarrow \mathcal{B}$ and is trivial on $\cB$. That is, it induces a $G$-action on each fiber of the deformation. If moreover $G$ is a Lie group, we will identify its Lie algebra with vector fields on each fiber of the deformation using the infinitesimal action. Note however that this identification depends a priori on the fiber. We 
can state our main result :

\begin{theointro}
\label{theo:defo}
Let $(X,L)=(X_{t_0},L_{t_0})$ be a polarized extremal K\"ahler manifold and $G_{t_0}^m$ be a maximal compact connected subgroup of $\Aut(X_{t_0},L_{t_0})$.
Let $\cL \rightarrow \cX \rightarrow \mathcal{B}$ be a polarized $G$-invariant deformation of $(X_{t_0},L_{t_0})$ with $G$ a compact connected subgroup of $G_{t_0}^m$.
Assume :
\begin{itemize}
\item[(1)]  $Lie(G)$ contains $J_{t_0}V_{t_0}^{G^m_{t_0}}$
\item[(2)] for some $t$ sufficiently close  to $t_0$ in $\mathcal{B}$, $Lie(G)$ contains $J_tV_t^{G^m_t}$,  with $G_t^m$ a maximal compact connected subgroup of $\Aut(X_t,L_t)$ such that $G\subset G_t^m \subset \Aut(X_t,L_t)$.
\end{itemize}
Then the modified K-energy  $E^{G_t^m}$  of $(X_t,L_t)$ is bounded from below.
\end{theointro}
\noindent This generalizes the result of Sz\'ekelyhidi, Chen and Tosatti for constant scalar curvature metrics to the extremal case.

\begin{rmkintro}
Note that hypothesis $(2)$ of Theorem \ref{theo:defo} is trivially satisfied if  $G$ is a maximal compact connected subgroup of $\Aut(X_t,L_t)$.
\end{rmkintro}

\begin{rmkintro}
If we assume that the deformation preserves a maximal compact subgroup of $\Aut(X,L)$, then the nearby fibers are automatically extremal, see \cite{acgt} or \cite{rst}. However if the deformation preserves a strictly smaller group with the hypothesis of Theorem \ref{theo:defo}, then a nearby fiber to $(X_{t_0},L_{t_0})$ must satisfy a stability condition to admit an extremal metric, see \cite{sz10}, \cite{br} and \cite{rt}. Nevertheless, by Theorem \ref{theo:defo}, the modified K-energy is bounded from below, even if the deformed manifold does not carry an extremal metric.  In Section~\ref{sec:ap}, we compute explicit examples of polarized $4$-points blow-ups of $\C\P^1\times\C\P^1$ that carry no extremal metrics, but with modified K-energy bounded from below.
\end{rmkintro}

The proof of Theorem \ref{theo:defo} follows the general lines of Tosatti's proof in the CSCK case. However some important technical points need to be generalized to the extremal setting. As these results are of independent interest, we state them below.
Let $(X,L)$ be a polarized extremal K\"ahler manifold and $G^m$ be a maximal compact connected subgroup of $\Aut(X,L)$.
Let $\cL \rightarrow \cX \stackrel{\pi}{\rightarrow} \mathcal{B}$ be a polarized $G$-invariant deformation of $(X,L)=(X_{t_0},L_{t_0})$ with $G$ a compact connected subgroup of $G^m$. Here we write $X_{t_0}=\pi^{-1}(t_0)$ and $L_{t_0}=\mathcal{L}|_{X_{t_0}}$, for $t_0\in\mathcal{B}$.
First, under the hypotheses of Theorem \ref{theo:defo}, we can build a special test configuration to simplify the problem (see Definition~\ref{def:testconf}). We extract from \cite{rt} the following result, which is due to Sz\'ekelyhidi in the CSCK case \cite{sz10}:
\begin{propintro}
\label{prop:test}
In the above situation, assume that hypothesis $(1)$ of Theorem \ref{theo:defo} is satisfied.
Then, for any $t\in\mathcal{B}$ sufficiently close to $t_0$ there is a smooth test configuration $\cL_T \rightarrow \cX_T \rightarrow \C$ with generic fibre $(X_t,L_t)$  satisfying :
\begin{itemize}
\item[(1)] the central fiber of the test configuration is a polarized extremal K\"ahler manifold $(\cX_{0},\cL_{0})$,
\item[(2)] the test configuration is $G$-invariant,
\item[(3)] ${J}_0V_0^{G_0^m}$ is contained in $Lie(G)$ where ${G}_0^m$ is a maximal compact connected subgroup of $Aut({\cX}_{0},{\cL}_{0})$ containing $G$.
\end{itemize}
\end{propintro}

To control the Mabuchi energy, we show a convexity inequality on the space of K\"ahler potentials.
Let $\cH$ be the space of K\"ahler potentials of the class $c_1(L)$ with respect to a fixed $G$-invariant metric $\omega\in c_1(L)$:
\begin{eqnarray*}
\cH= \lbrace \phi \in C^{\infty} (X) \vert \om_{\phi}:=\om + \sqrt{-1}\del\delb \phi > 0 \rbrace.
\end{eqnarray*}
Denote  $\cH^G$ the space of invariant potentials under the $G$-action. Then the Calabi functional and Mabuchi functional admit $G$-invariant relative versions $Ca^G$ and $E^G$ on $\cH^G$  (see Section \ref{sec:modified}). Then the following is a generalization of Chen's inequality \cite{ch}:
\begin{propintro}
\label{prop:convex} 
For any $\phi_0$, $\phi_1$ in $\cH^G$, we have
\begin{eqnarray*}
E^G(\phi_1)-E^G(\phi_0)\leq d(\phi_0,\phi_1) \cdot \sqrt{Ca^G(\phi_1)}.
\end{eqnarray*}
\end{propintro}

\noindent Note that in \cite{ch}, Chen makes no polarization assumption, while our method to obtain this inequality is by quantization, following Chen and Sun \cite{cs}. Together with Proposition \ref{prop:test}, Proposition \ref{prop:convex} enables us to prove Theorem \ref{theo:defo}.

\subsection{Plan of the paper}
We start with definitions of extremal metrics, modified K-energy and relative Futaki invariant in Section \ref{sec:background}. Section \ref{sec:convex} is devoted to the proof of Proposition \ref{prop:convex} using quantization. The proof of Proposition \ref{prop:test} is done in Section \ref{sec:test}. We shall mention that Section \ref{sec:convex} and Section \ref{sec:test} are independent. Lastly, we prove Theorem \ref{theo:defo} in section \ref{sec:prooftheo} and study an application in Section \ref{sec:ap}.

\subsection{Acknowledgments}
The authors would like to thank Song Sun and Valentino Tosatti for their encouragement and comments. They are also grateful to Vestislav Apostolov and Gabor Sz\'ekelyhidi for useful discussions. This work was largely undertaken while the first author was a post-doctoral fellow at the Universidade de S\~ao Paulo. He would also like to acknowledge the support of FAPESP processo 2011/07363-6 and Claudio Gorodski for many fruitful mathematical discussions. The second author is supported by a CRM-ISM grant and would like to acknowledge CIRGET for providing a stimulating mathematical environment.

\section{Extremal metrics}
\label{sec:background}
We define extremal metrics in this section, and collect some standard facts about the modified Mabuchi functional and the relative Futaki invariant that will be used in the paper. 
\subsection{Definition}
Let $(X,L)$ be a polarized K\"ahler manifold of complex dimension $n$. Let $\cH$ be the space of smooth K\"ahler potentials with respect to a fixed K\"ahler form $\om \in c_1(L)$  :
\begin{eqnarray*}
\cH= \lbrace \phi \in C^{\infty} (X) \vert \om_{\phi}:=\om + \sqrt{-1}\del\delb \phi > 0 \rbrace
\end{eqnarray*}
In order to find a canonical representative of a K\"ahler class, Calabi \cite{c1} suggested considering the functional
\begin{eqnarray*}
 Ca :  \cH & \to & \R \\
  \phi & \mapsto & \int_X (S(\phi)-\uS)^2 d\mu_{\phi}
\end{eqnarray*}
where $S(\phi)$ is the scalar curvature of the metric $g_{\phi}$ associated to the K\"ahler form $\om_{\phi}$, $$\uS=2n\pi\frac{c_1(L)\cup [\om]^{n-1}}{[\om]^n}$$ is the average of the scalar curvature, an invariant of the K\"ahler class, and $d\mu_{\phi}=\dfrac{\om_{\phi}^n}{n!}$ the volume form of $g_{\phi}$.
The Hessian of $Ca$ at a critical point is positive, and the local minima are called {\it extremal metrics}. 
The associated Euler-Lagrange equation is equivalent to the fact that $grad_{\om_{\phi}}(S(\phi))$ is a holomorphic vector field.
In particular, constant scalar curvature K\"ahler metrics, CSCK for short, are extremal metrics.
\par
By a result of Calabi \cite{c2}, the connected component of the identity of the isometry group of an extremal metric is a maximal compact connected subgroup of the reduced automorphism group $\Aut_0(X)$. Note that the latter group is isomorphic to the connected component of identity of $\Aut(X,L)$.This is the motivation for working modulo a maximal compact subgroup of $\Aut(X,L)$ when dealing with extremal metrics. However complex deformations do not in general preserve such symmetries, so we will instead work modulo any connected compact subgroup of $\Aut(X,L)$ and define the relevant functionals in this case. Let $G$ be a compact connected subgroup of $\Aut(X,L)$. We assume now that $\om$ is $G$-invariant and denote $\cH^G$ the space of $G$-invariant potentials.

\subsection{Modified K-energy}
\label{sec:modified}

For a fixed $G$-invariant K\"ahler metric $g_{\phi}$, we say that a vector field $V$ is a hamiltonian vector field 
if there is a real valued function $f$ such that
\begin{eqnarray*}
V=J\nabla_{g_{\phi}} f.
\end{eqnarray*}
If in addition $V$ is Killing, we say that the function $f$ is a Killing potential.
Let $\g$ be the Lie algebra of $G$.
For any $\phi\in\cH^G$, let $P_{\phi}^G$ be the space of normalized (i.e. of mean value zero) Killing potentials with respect to $g_{\phi}$ 
whose corresponding hamiltonian vector field lies in $\g$ and let $\Pi_{\phi}^G$ be the orthogonal projection from $L^2(X,\R)$ to $P_{\phi}^G$ given by the inner product on functions 
$$
(f,g) \mapsto \int fg d\mu_{\phi}.
$$
Note that $G$-invariant metrics satisfying $S(\phi)-\uS-\Pi_{\phi}^G S(\phi)=0$ are extremal.

\begin{definition}\cite[Section 4.13]{gbook}
The reduced scalar curvature $S^G$ with respect to $G$ is defined by
\begin{eqnarray*}
S^G(\phi)=S(\phi)-\uS-\Pi_{\phi}^G S(\phi).
\end{eqnarray*}
The extremal vector field $V^G$ with respect to $G$ is defined by the equation
$$
V^G=\nabla_g (\Pi_\phi^G S(\phi))
$$
for any $\phi$ in $\cH^G$ and does not depend on $\phi$ (see for example \cite[Proposition 4.13.1]{gbook}).
\end{definition}

\begin{remark}
Note that by definition the extremal vector field relative to $G$ is real-holomorphic and lies in $J\g$ where $J$ is the almost-complex structure of $X$, while $JV^G$ lies in $\g$.
\end{remark}

\begin{remark}
When $G=\lbrace 1 \rbrace$ we recover the normalized scalar curvature. When $G$ is a maximal compact connected subgroup, or maximal torus of $\Aut_0(X)$, we find the reduced scalar curvature
and the extremal vector field initially defined by Futaki and Mabuchi \cite{fm}. The extremal vector field only depends on the K\"ahler class and the choice of the maximal compact connected subgroup of $\Aut(X,L)$.
\end{remark}

\par
The relative Mabuchi K-energy was introduced by Guan \cite{gu}, Chen and Tian \cite{ct}, and Simanca \cite{si}:

\begin{definition}\cite[Section 4.13]{gbook} 
\label{def:modK}
The modified K-energy or modified Mabuchi energy (relative to $G$) $E^G$ is defined, up to a constant, as the primitive of the following one-form on $\cH^G$:
$$
\phi \mapsto - S^G(\phi) d\mu_{\phi}.
$$
\end{definition}

\noindent If $\phi\in \cH^G$, then the modified K-energy relative to $G$ admits the following expression
$$
E^G(\phi)=-\int_X \phi (\int_0^1S^G(t\phi) d\mu_{t\phi} dt) .
$$
As for CSCK metrics, $G$-invariant extremal metrics whose extremal vector field lies in $J\g$ are critical points of the modify K-energy $E^G$.

An important point that we will use several times in the sequel is the following remark:

\begin{remark}
\label{rmk:mab}
The modified K-energy $E^{G^m}$ is defined to be the modified K-energy with respect to a maximal compact connected subgroup $G^m$ of $\Aut(X,L)$. Let $G^m$ be such a maximal compact connected group, and let $G$ be a compact connected subgroup of $G^m$. Assume that $Lie(G)$ contains the extremal vector field of $(X,c_1(L))$ with respect to $G^m$. Then $E^{G^m}$ is equal to $E^G$ when restricted to the space of $G^m$-invariant potentials.
Indeed, the projection of any $G^m$-invariant scalar curvature to the space of Killing potentials of $Lie(G^m)$ gives a potential for the extremal vector field by definition.
Thus a minimiser of $E^G$ that is invariant under the $G^m$-action, such as an extremal metric, will be a minimum of the standard modified Mabuchi Energy.
\end{remark}

The Calabi energy can also be generalized :

\begin{definition}
The modified Calabi functional $Ca^G$ with respect to $G$ is defined on $\cH^G$ by
$$
Ca^G(\phi)=\int_X (S^G(\phi))^2 d\mu_{\phi}.
$$
\end{definition}

\subsection{Relative Futaki invariant}
Let $\h_0$ be the Lie algebra of $\Aut(X,L)$. We want to work modulo the $G$-action, so let $\h^G$ be the Lie algebra of the normalizer of $G$ in $\Aut(X,L)$. Introduced by Futaki as an obstruction to the existence of K\"ahler-Einstein metric \cite{fut}, the Futaki character has been generalized to any K\"ahler class and admits a relative version.
\noindent It is well known that for any K\"ahler metric $g_{\phi}$ representing the K\"ahler class $c_1(L)$, each element $V$ of $\h_0$ can be uniquely written
$$
V=\nabla_{g_\phi}(f_\phi^V)+J\nabla_g{_\phi}(h_\phi^V)
$$
where $f_\phi^V$ and $h_\phi^V$ are real valued functions on $X$, normalized to have mean value zero (see e.g. \cite[Lemma 2.1.1]{gbook}). We will call $f_\phi^V$ the {\it real potential} of $V$ with respect to $g_\phi$.

\begin{definition}\cite[Defn. 7]{gbook}\label{defn:FutChar}
The Futaki character relative to $G$, denoted $\cF^G$, is defined by :
\begin{eqnarray*}
\mathcal{F}^G :  \h^G/\g & \rightarrow & \R \\
  V & \mapsto & \int_X f_\phi^V S^G(\phi) d\nu_{\phi}
 \end{eqnarray*}
with $\phi\in\cH^G$.
\end{definition}

We sum-up some properties of this invariant that we shall need in the sequel of this paper. The proof of these facts are due to Futaki (see e.g. \cite{futbook}), and a proof of their relative versions can be found in \cite{gbook}.

\begin{proposition}
The Futaki invariant relative to $G$ does not depend on the choice of the K\"ahler metric $g_\phi$ for $\phi\in\cH^G$ and is well defined. If there is an extremal metric on $X$ in the K\"ahler class $c_1(L)$ whose extremal vector field lies in $J\g$, then $\mathcal{F}^G$ is identically zero.
\end{proposition}

\section{A convexity inequality on $\cH^G$ via quantization}
\label{sec:convex}
The aim of this section is the proof of Proposition~\ref{prop:convex}.

\subsection{Quantization: the space of potentials}
For each $k$, we can consider the space $\cH_k$ of hermitian metrics on $L^{\otimes k}$ with positive curvature. To each element $h\in \cH_k$ one associates a metric $\omega_h= -\sqrt{-1} \del\delb log(h)$ on $X$, thereby identifying the spaces $\cH_k$ and $\cH$. 
Fixing a base metric $h_0$  in $\cH_1$ such that $\om=\om_{h_0}$ the correspondence reads
\begin{eqnarray*}
\om_{\phi}=\om_{e^{-\phi}h_0}=\om+\sqrt{-1}\del\delb \phi .
\end{eqnarray*}
We denote by $\cB_k$ the space of positive definite Hermitian forms on $H^0(X,L^{\otimes k})$.
The spaces $\cB_k$ are identified with $GL_{N_k}(\C)/ U(N_k)$, using the base metric $h_0^k$ and where $N_k$ is the dimension of $H^0(X,L^k)$. These symmetric spaces come with metrics $d_k$ defined by  Riemannian metrics:
\begin{eqnarray*}
(H_1,H_2)_h=Tr(H_1H^{-1}\cdot H_2 H^{-1}).
\end{eqnarray*}

\noindent There are maps :
\begin{eqnarray*}
Hilb_k :  \cH & \rightarrow &\cB_k \\
FS_k : \cB_k &\rightarrow &\cH
\end{eqnarray*}
defined by :
\begin{eqnarray*}
\forall h\in \cH\;, \; s\in H^0(X,L^{\otimes k})\;, \; \vert\vert s\vert\vert ^2_{Hilb_k(h)}=\int_X \vert s \vert_h^2 d\mu_h
\end{eqnarray*}
and
\begin{eqnarray*}
\forall H \in \cB_k\; , \; FS_k(H)= \frac{1}{k} \log \sum_{\alpha} \vert s_{\alpha}\vert_{h_0^k}^2
\end{eqnarray*}
where  $\lbrace s_{\alpha}\rbrace$ is an orthonormal basis of $H^0(X,L^{\otimes k})$ with respect to $H$. Note that $\om_{FS_k(H)}$ is the pull-back of the Fubini-Study metric on 
$\P(H^0(X,L^k)^*)$ that is induced by the inner product $H$ on $H^0(X,L^k)$. 
A result of Tian \cite{tian90} states that any K\"ahler metric $\om_{\phi}$ in $c_1(L)$ can be approximated by projective metrics, namely
\begin{eqnarray*}
\lim_{k\rightarrow \infty} \frac{1}{k} FS_k \circ Hilb_k (\phi) = \phi
\end{eqnarray*}
where the convergence is uniform on $C^2(X,\R)$ bounded subsets of $\cH$.
\noindent 
Let $G$ be a compact connected subgroup of $\Aut(X,L)$.
We can assume that $\Aut(X,L)$ acts on $L$, considering a sufficiently large tensor power if necessary (see e.g. \cite{kob}).
Then the $G$-action on $X$ induces a $G$-action on the space of sections $H^0(X,L^k)$. This action in turn provides a $G$-action on the space  $\cB_k$  of positive definite hermitian forms on $H^0(X,L^k)$ and we define $\cB_k^G$ to be the subspace of $G$-invariant elements.
Note that the spaces $\cB_k^G$ are totally geodesic in $\cB_k$ for the distances $d_k$.
There are the induced maps :
$$
\begin{array}{cccc}
Hilb_k : & \cH^G & \rightarrow &\cB_k^G \\
FS_k :& \cB_k^G &\rightarrow &\cH^G.
\end{array}
$$
By a result of Chen and Sun \cite{cs}, the metric spaces $(\cB_k,d_k)$ converge to $(\cH,d)$ where $d$ is the Weyl-Petersson metric given by
$$
(\delta\phi_1,\delta \phi_2)_{\phi} = \int_X \delta\phi_1 \delta\phi_2 d\mu_{\phi}.
$$
Consider the induced Weyl-Petersson metric on $\cH^G$ and the associated distance function $d_{\cH^G}$. The space $\cH^G$ is a totally geodesic subspace of $\cH$ for the distance $d$ as a the set of fixed points by an isometry group. Thus the following results is a direct consequence of the work of Chen and Sun \cite{cs} :

\begin{theorem}\cite[Thm. 1.1]{cs}
\label{theo:CS}
Given any $\phi_0$, $\phi_1$ in $\cH^G$, we have
\begin{eqnarray*}
\lim_{k\rightarrow \infty} k^{-\frac{n+2}{2}} d_k(Hilb_k(\phi_0),Hilb_k(\phi_1))=d_{\cH^G}(\phi_0,\phi_1).
\end{eqnarray*}
\end{theorem}

\subsection{Quantization of extremal metrics}

In order to find a finite dimensional approximation of extremal metrics, 
Sano introduced the $\sigma$-balanced metrics (see  \cite{st} for a first application of these metrics):

\begin{definition}
Let $\sigma_k(t)$ be a one-parameter subgroup of $\Aut(X,L^k)$. A metric $\om_\phi$ is called a $\sigma_k$-balanced metric if 
$$
\om_{kFS_k\circ Hilb_k(\phi)}=\sigma_k(1)^*\om_{k\phi}
$$
\end{definition}

Conjecturally, the $\sigma$-balanced metrics would approximate an extremal K\"ahler metric and generalize Donaldson's results \cite{Don01} and Mabuchi's work \cite{ma04}. Indeed, in one direction, assume that we are given $\sigma_k$-balanced metrics $\om_{\phi_k}$, with $\sigma_k(t)\in \Aut(X,L^k)$ such that the $\om_k$ converge to $\om_{\infty}$.
Suppose that the vector fields $k\frac{d}{dt}\vert_{t=0}\sigma_k(t)$ converge to a vector field $V_{\infty}\in \h_0$. 
Then a simple calculation implies that $\om_{\infty}$ must be extremal.
Let $\sigma$ be a one parameter subgroup of $\Aut(X,L)$ generated by a vector field $V$ and consider the normalized vector fields $V_k=-\frac{V}{4k}$ and the associated one-parameter groups $\sigma_k$. 
Define for each $\phi\in \cH$ the functions $\psi_{\sigma_k,\phi}$ by
$$
\sigma_k(1)^*\om_{\phi}=\om_{\phi}+\sqrt{-1} \del\delb \psi_{\sigma_k,\phi}.
$$
normalized by
\begin{eqnarray*}
\forall k\ \ \ \ \; \int_X \exp{(\psi_{\sigma_k,\phi})}\; d\mu_{\phi} = \frac{N_k}{k^n}.
\end{eqnarray*}
Define $I_k=\log\circ \det$ on $\cB_k$. This functional is defined up to an additive constant when we see $\cB_k$ as a space of positive Hermitian matrix
once a suitable basis of $H^0(X,L^k)$ is fixed. 
Then we define for each $k$
\begin{eqnarray*}
\delta I^{\sigma}_k(\phi)(\delta\phi)=\int_X k\delta \phi (1+\frac{\Delta_{\phi}}{k})e^{\psi_{\sigma_k,\phi}} k^n d\mu_{\phi}
\end{eqnarray*}
where $\Delta_{\phi}=-g_{\phi}^{i\overline{j}}\frac{\del}{\del z_i}\frac{\del}{\del \overline{z}_j}$ is the complex Laplacian of $g_{\phi}$.

\begin{remark}
This one-form integrates along paths in $\cH$ to a functional $I_k^{\sigma}(\phi)$ on $\cH$, which is independent on the path used from $0$ to $\phi$ \cite{st}.
\end{remark}

Then we define $Z^{\sigma}_k$ on $\cB_k$ by
\begin{eqnarray*}
Z^{\sigma}_k =  I_k^{\sigma}\circ FS_k + I_k-k^n\log(k^n)V.
\end{eqnarray*}

\begin{remark}
The definition of the functionals $I_k^{\sigma}$ and $Z_k^{\sigma}$ is motivated by Donalsdon's work in the CSCK case \cite{Don05}.
\end{remark}

\noindent Let $G^m$ be a maximal compact subgroup of $\Aut(X,L)$. Let $G$ be a compact connected subgroup of $G^m$ such that $JV^{G_m}$ is contained in its Lie algebra.
By a theorem of Futaki and Mabuchi \cite{fm}, the vector field $JV^{G^m}$ generates a periodic action by a one parameter-subgroup of automorphisms of $(X,L)$. We fix $\sigma(t)$ to be this one-parameter group.
In that case, the functionals $Z_k^{\sigma}$ approximate the modified Mabuchi functional \cite{st}. We will use the following results :

\begin{proposition}\cite{st}
\label{prop:Zk}
The functional $Z^{\sigma}_k$ is convex along geodesics in $\cB_k^G$.
\par
\noindent There are constants $c_k$ such that 
$$
\frac{2}{k^n}Z_k^{\sigma}\circ  Hilb_k + c_k \rightarrow E^G
$$
as $k\rightarrow \infty$, where the convergence is uniform on $C^l(X,\R)$ bounded subsets of $\cH^G$.
\end{proposition}

\begin{remark}
The choice of the group $G$ is more general here than in \cite{st} where the computations are done modulo the one-parameter group generated by $\sigma(t)$. Let $G_e$ denote this group. Then the proof of Proposition~\ref{prop:Zk} only uses the choice of $\sigma(t)$ in the definition of $Z_k^{\sigma}$ and the fact that all the considered tensors are $G_e$-invariant. Using that in our situation $E^G=E^{G_e}$ by Remark~\ref{rmk:mab}, the results from \cite{st} extend here.
\end{remark}

\begin{remark}
There is no reason for $V_k=\frac{d}{dt}\vert_{t=0}\sigma_k(t)$ to be the right quantization of the extremal vector field.
The choice of the one-parameter subgroups $\sigma_k$ made in the above proposition is certainly not the appropriate one if we want to show that $\sigma$-balanced metrics approximate extremal metrics. However it will be sufficient for our purposes.
\end{remark}

\subsection{Proof of proposition~\ref{prop:convex} }

As in the previous section, $G$ denotes a compact connected subgroup of $\Aut(X,L)$ contained in some maximal compact group $G^m\subset\Aut(X,L)$ and containing in its Lie algebra the extremal vector field $JV^{G^m}$. The one parameter subgroups generated by the vector fields $-\frac{V^{G^m}}{4k}$ are denoted $\sigma_k(t)$.

The following inequality is a generalization of a result of Chen \cite{ch}:

\begin{proposition}
\label{prop:convx}
For any $\phi_0$, $\phi_1$ in $\cH^G$, 
\begin{eqnarray*}
E^G(\phi_1)-E^G(\phi_0)\leq d(\phi_0,\phi_1) \cdot \sqrt{Ca^G(\phi_1)}
\end{eqnarray*}
where $E^G$ is the modified K-energy and $Ca^G$ the modified Calabi energy.
\end{proposition}

\noindent The following result will be useful :

\begin{theorem}[\cite{cat,ruan,tian90,zel}]
\label{theo:bergman}
Let
$$
\rho_k(\phi)=\sum_{\alpha}\vert s_{\alpha}\vert^2_{h^k}
$$
be the Bergman function of $\phi\in \cH$, where $h=e^{-\phi}h_0$. 
The following uniform expansion holds
\begin{eqnarray*}
\rho_k(\phi)=k^n+A_1(\phi)k^{n-1}+A_2(\phi)k^{n-2}+...
\end{eqnarray*}
with $A_1(\phi)=\frac{1}{2}S(\phi)$
and for any $l$ and $R\in \N$, there is a constant $C_{l,R}$ such that
$$
\vert\vert\rho_k(\phi) -\sum_{j\leq R} A_j k^{n-j} \vert \vert_{C^l} \leq C_{l,R} k^{n-R}.
$$
\end{theorem}

\noindent We will also use the two following lemmas :

\begin{lemma}
\label{lem:exppsi}\cite{st}
Let $\psi_k(\phi)=\psi_{\sigma_k,\phi}$.
The following expansion holds uniformly in $C^l(X,\R)$ for $l\gg 1$:
$$
\psi_k(\phi)=\frac{\Pi_{\phi}^G S(\phi)+\uS}{2k}+\mathcal{O}(k^{-1}).
$$
\end{lemma}
 
\begin{lemma}
\label{lem:gradZ}
Let $\phi\in\cH^G$, and let $H_k=Hilb_k(\phi)$.
Then 
\begin{equation}
\label{eq:gradZ}
\lim_{k\rightarrow \infty} k^{-n+2} \vert\vert\nabla Z_k^{\sigma}(H_k)\vert\vert^2 =\frac{1}{4} Ca^G(\phi)
\end{equation}
\end{lemma}

\begin{remark}
The proof of Lemma~\ref{lem:exppsi} is the same as the one in \cite{st}. Again, we use Remark \ref{rmk:mab} which implies that for $G$-invariant potentials $\phi$, $\Pi_{\phi}^G S(\phi)=\Pi_{\phi}^{G_e} S(\phi)$ where $G_e$ is the group generated by $JV^{G^m}$.
\end{remark}

\begin{proof}[Proof of Proposition~\ref{prop:convx}]
We follow the proof of Chen and Sun \cite{cs} where $G=\lbrace 1 \rbrace$.
By Proposition~\ref{prop:Zk}, the functional $Z^{\sigma}_k$ is convex along the geodesic in $\cB_k^G$ joining $H_k^0=Hilb_k(\phi_0)$ to $H_k^1=Hilb_k(\phi_1)$, thus
$$
Z^{\sigma}_k(H^1_k)-Z^{\sigma}_k(H_k^0)\leq d_{\cB_k}(H_k^0,H_k^1) \cdot \vert\vert \nabla Z_k^{\sigma}(H_k^1)\vert\vert.
$$
Then again by Proposition~\ref{prop:Zk}, 
$$
\lim_{k\rightarrow \infty} k^{-n}(Z_k^{\sigma}(H_k^1)-Z_k^{\sigma}(H_k^0))= \frac{1}{2}(E^G(\phi_1)-E^G(\phi_0)),
$$
and by Lemma~\ref{lem:gradZ}
$$
\lim_{k\rightarrow \infty} k^{-n+2} \vert\vert\nabla Z_k^{\sigma}(H_k^1)\vert\vert^2 = \frac{1}{4}Ca^G(\phi_1).
$$
Then the proof follows from the Theorem ~\ref{theo:CS}, when $k$ goes to infinity.
\end{proof}

\noindent We conclude this section with the proof of Lemma~\ref{lem:gradZ}.

\begin{proof}[Proof of Lemma~\ref{lem:gradZ}]
Let $\phi_k=FS_k(H_k)$.
To compute the left hand side of Equation~(\ref{eq:gradZ}), let's first compute its differential
\begin{eqnarray*}
\delta (Z_k^{\sigma})_H(\delta H) &=& \delta I_k^{\sigma}\circ FS_k (\delta H) + \delta I_k (\delta H)\\
 & = & k^n \int_X (k+\Delta_{FS_k(H)})e^{\psi(FS_k(H))}(\delta FS_k(\delta H)) d\mu_{FS_k(H)} + trace(\delta H) \\
\end{eqnarray*}
where $\lbrace s_i\rbrace$ is an orthonormal basis of $H$.
As
$$
\delta FS_k(\delta H)=-\frac{1}{k} \sum_{i,j} \delta H_{i,j}\cdot (s_i,s_j)_{FS_k(H)}
$$
We obtain
$$
(\nabla Z_k^{\sigma})_{i,j}(H) = -k^n \int_X (1+\frac{\Delta_{FS_k(H)}}{k})e^{\psi(FS_k(H))} (s_i,s_j)_{FS_k(H)} d\mu_{FS_k(H)} + \ep_{i,j}
$$
with $\ep_{i,j}=1$ if $i=j$ and $0$ in the other cases. With no restriction we can assume $[\nabla Z_k^{\sigma}]$ to be diagonal. Then evaluate at $H_k$:

\begin{equation}
\label{eq:Z}
(\nabla Z_k^{\sigma})_{i,i}(H_k) = -k^n \int_X (1+\frac{\Delta_{\phi_k}}{k})e^{\psi(\phi_k)} \vert s_i\vert^2_{\phi_k} d\mu_{\phi_k} + 1
\end{equation}

\noindent From the expansion of Bergman kernel in Theorem~\ref{theo:bergman} we deduce
\begin{equation}
\label{eq:exp1}
\vert s_i \vert^2_{\phi_k}=k^{-n}\vert s_i \vert^2_{\phi}(
1-\frac{S(\phi)}{2k}+\cO(k^{-2}))
\end{equation}
and
\begin{equation}
\label{eq:exp2}
\om_{\phi_k}=\om_{\phi}(1+\cO(k^{-2})).
\end{equation}

\noindent From Lemma~\ref{lem:exppsi} we also have the uniform expansion :

\begin{equation}
\label{eq:exp3}
\psi_k(\phi)=\frac{\Pi_{\phi}^G S(\phi)+\uS}{2k}+\cO(k^{-1})
\end{equation}

\noindent Then, Equations (\ref{eq:exp1}), (\ref{eq:exp2}) and (\ref{eq:exp3}) together with (\ref{eq:Z}) imply
$$
(\nabla Z_k^{\sigma})_{i,i}(H_k) = 1 - \int_X \vert s_i\vert _{\phi}^2 d\mu_{\phi} + \frac{1}{2k} \int_X S^G(\phi) \vert s_i \vert_{\phi}^2 d\mu_{\phi} + \cO(k^{-1}).
$$
As 
$$
\int_X  \vert s_i \vert_{\phi}^2 d\mu_{\phi}=1
$$
we end with
$$
(\nabla Z_k^{\sigma})_{i,i}(H_k) =\frac{1}{2k} \int_X S^G(\phi) \vert s_i \vert_{\phi}^2 d\mu_{\phi} + \cO(k^{-1}).
$$
Then the proof follows from the following fact (see \cite[Rmk. 3.3]{cs}) :
$$
\lim_{k\rightarrow \infty} k^{-n} \sum_{i,j}\vert \int_X (s_i,s_j)_{\phi} \psi d\mu_{\phi} \vert^2=\int_X \psi^2 d\mu_{\phi}
$$
for any $\psi\in \cH^G$.
\end{proof}


\section{Invariant Deformations}
\label{sec:test}

In this section we give the requisite definitions and terminology necessary to prove Theorem \ref{theo:defo}. In particular we will discuss deformations of polarized complex manifolds that are invariant under a compact group action. We will consider $M$ to denote a smooth real manifold, and $X$ to denote the complex manifold $(M,J)$ when $M$ is equipped with an integrable complex structure $J$.  We also assume that the K\"ahler structure comes from a polarization $L\rightarrow X$.


\begin{definition}
\label{def:deformation}
Let $(X,L)$ be a polarized complex manifold. A {\it polarized deformation} of $(X,L)$ is a triple of complex manifolds $(\cB,\cX,\cL)$, with a fixed point $t_0\in\mathcal{B}$, together with holomorphic maps $\cL\rightarrow \cX \stackrel{\pi}{\rightarrow} \cB$ such that
\begin{itemize}
\item $\pi :\cX \rightarrow \cB$ is a proper submersion,
\item $\cL$ is a holomophic line bundle over $\cX$ so that the restriction $L_t$ to the fibre $X_t=\pi^{-1}(t)$ is ample,
\item $(X,L)$ is isomorphic to $(X_{t_0},L_{t_0})$.
\end{itemize}
Given a compact Lie group $G$, the deformation is $G$-\emph{invariant} if $\cL$ and $\cX$ are acted upon by $G$, compatibly with the projection $\cL\to\cX$ and inducing the identity action on $\cB$.
\end{definition}

 We will also consider test-configurations:


\begin{definition}
\label{def:testconf}
Let $(X,L)$ be a polarized complex manifold. A {\it smooth test-configuration} for $(X,L)$ is given by the following data :
\begin{itemize}
\item a proper holomorphic submersion $\cX_T\stackrel{\pi}{\rightarrow}\C$ and a line bundle $\cL_T\to\cX_T$ that restricts to an ample bundle $\cL_z$ on each fibre $\cX_z=\pi^{-1}(z)$,
\item for each $z\neq 0\in\C$, the pair $(\mathcal{X}_z,\cL_z)$ is isomorphic to $(X,L)$,
\item the group $\C^*$ acts on the pair $(\cX_T,\cL_T)$ so as to induce the $\C^*$-action by multiplication on $\C$.
\end{itemize}
Given a compact Lie group $G$, we will say that the test configuration is {\it $G$-invariant} if it is a $G$-invariant deformation of its central fiber and the $G$-action commutes with the $\C^*$-action.
\end{definition}

As in previous sections, if $G^m$ is a maximal connected compact subgroup of reduced automorphisms of the polarized variety $(X,L)$, we denote by $V^{G^m}$ the extremal vector field on $X$, for the K\"ahler class $c_1(L)$ and the group $G^m$.
 If $X_t$ is a fibre of a deformation, 
 for $t\in \mathcal{B}$, we denote a maximal compact connected subgroup of $Aut(X_t,L_t)$ by $G^m_t$.

 We wish to prove, following Sz\'ekelyhidi \cite{sz10} and Rollin and the second author \cite{rt}, that if $(X,L)$ admits an extremal metric in $c_1(L)$, then any nearby fibre in a invariant deformation of $(X,L)$ can be taken to be the generic fibre of an invariant test configuration, the central fibre of which is also extremal.

Let $(\cB,\cX,\cL)$ be a polarized deformation of the smooth polarized variety $(X,L)$.
The fibration $\mathcal{X}\to \mathcal{B}$ is smoothly trivial 
so there exists a diffeomorphism (perhaps for smaller $\mathcal{B}$),
\begin{eqnarray*}
F:M\times \mathcal{B}\to \mathcal{X}
\end{eqnarray*}
with respect to which the deformations can be considered a family $(M,J_t)$ for $t\in \mathcal{B}$. Also, the group action of $G$ on $\mathcal{X}$ can be considered a map, for $t\in\mathcal{B}$, 
\begin{eqnarray*}
\sigma_t:G\times M\to M.
\end{eqnarray*}
By a theorem of Palais and Stewart (see \cite{ps}) on the rigidity of compact group actions, there exists a smooth family of diffeomorphisms $f_t$ such that 
\begin{eqnarray*}
f_t(\sigma_t(g,f_t^{-1}(x)))=\sigma_0(g,x).
\end{eqnarray*}
That is, we can amend $F$ so that the action of $G$ on $\mathcal{X}$ can be considered an action on $M$ that is independent of $t$. By also adjusting the complex structures $J_t$ by a diffeomorphism, we can suppose that the action is holomophic with respect to each complex structure $J_t$. We can also suppose that $L_t$ is a complex line bundle on $M$ with $c_1(L_t)=c_1(L_0)$ fixed. $J_t$ is a $G$-invariant complex structure on $M$, compatible with K\"ahler form $\omega_t\in c_1(L_t)$. By Moser's theorem, $\omega_t$ is equivalent, by some $G$-invariant diffeomorphism, to $\omega_0$. We can then suppose that all complex structures $J_t$ are compatible with a fixed symplectic form $\omega$ on $M$. 

The deformation can then be considered a smooth map from $\mathcal{B}$ to the set $\mathcal{J}^G$ of  G-invariant almost-complex structures on $M$ that are compatible with the fixed symplectic form $\omega$. 


We first recall that the hermitian scalar curvature $S(J)$ of the metric $g=(\omega,J)$, for $J\in\mathcal{J}^G$, is given by the trace with respect to $\omega$ of the curvature of the Chern connection on the anti-canonical bundle $K^*_X$. We suppose that $G$ acts by hamiltonian diffeomorphisms of $(M,\omega)$. The reduced scalar curvature of the hermitian metric $g=(\omega,J)$, with respect to the group $G$, is given by
\begin{eqnarray*}
S^G(J)=S(J)-\Pi^G_\omega\left( S(J)\right)
\end{eqnarray*}
where we recall the projection $\Pi^G_\omega$ on the space of hamiltonian killing fields is defined in Section \ref{sec:modified}.

 
 Let $H$ be the group of hamiltonian diffeomorphisms of $(M,\omega)$, let $Z(G,H)$ be the centraliser of $G$ in $H$ and let $\mathcal{G}=Z(G,H)/(Z(G,H)\cap G)$ be the quotient by the centre of $G$. 
 The set $\mathcal{J}^G$ admits the structure of an infinite dimensional K\"ahler manifold \cite{fuj} and $\mathcal{G}$ acts on $\mathcal{J}^G$ by automorphisms of this structure.We denote by $\Omega$ the K\"ahler form on $\mathcal{J}^G$. 
 By considering hamiltonian potentials, the Lie algebra of $\mathcal{G}$ can be identified with those smooth $G$-invariant functions of $\omega$-mean equal to $0$, that are $L^2$-orthogonal to $P^G_\omega$.   This can be identified by $L^2$-inner product with its dual space. 

\begin{theorem} \emph{\cite{fuj,Don97,gbook}}
The action of $\mathcal{G}$ on $\mathcal{J}^G$ is hamiltonian, and admits a $\mathcal{G}$-equivariant moment map given by
\begin{eqnarray*}
\mu^G:\mathcal{J}^G& \to & C^\infty_0(M,\R)^G\\
J &\mapsto & S^G(J).
\end{eqnarray*}
\end{theorem}

For $J\in\mathcal{J}^G$, the tangent space to $\mathcal{J}^G$ at $J$ is given by 
\begin{eqnarray*}
T_J\mathcal{J}^G=\{ \alpha\in\Omega^{0,1}(T^{1,0}X)^G\ ; \ \omega(\alpha(\cdot),\cdot)+\omega(\cdot,\alpha(\cdot))=0\}.
\end{eqnarray*}

The set of infinitesimal $G$-invariant deformations is given by the kernel of the operator
$\bar\partial:\Omega^{0,1}(T^{1,0})^G\to \Omega^{0,2}(T^{1,0})^G$.
The infinitesimal action of $\mathcal{G}$ at $J$ is given by 
\begin{eqnarray*}
P:C^\infty_0(X,\R)^G&\to & \Omega^{0,1}(T^{1,0}X)^G,\\
f&\mapsto & \bar\partial v_f^{(1,0)}
\end{eqnarray*}
where $v_f^{(1,0)}$ is the $(1,0)$-part of the hamiltonian vector field associated to the potential $f$. Following \cite{Don97} we consider the complexified orbits of the $\mathcal{G}$-action on $\mathcal{J}^G$. The operator $P$ can be complexified, so as to obtain $P:C^\infty_0(X,\C)^G\to\Omega^{0,1}(T^{1,0}X)^G$. Together these operators define an elliptic complex 
\begin{eqnarray*}
C^\infty_0(X,\C)^G \stackrel{P}{\longrightarrow} \Omega^{0,1}(T^{1,0})^G \stackrel{\bar\partial}{\longrightarrow} \Omega^{0,2}(T^{1,0})^G.
\end{eqnarray*}

 The images of the operator $P$, as $J$ varies, define an integrable distribution on $\mathcal{J}^G$, and the maximal integral submanifolds are the complexified orbits of the action of $\mathcal{G}$. 
 The complexified orbits have particular relevance because (see \cite{Don97}) if $J$ and $J^\prime$ lie in the same complexified orbit, and $J$ is integrable, then the pair $(\omega, J^\prime)$ is equivalent, via some diffeomorphism, to $(\omega +i\partial\bar\partial \psi, J)$. 
 
From this point we fix an integrable complex structure $J_0$ such that $X=(M,J_0)$ and suppose that $\mu^G(J_0)=S^G(J_0)=0$. That is, $g_0=(\omega,J_0)$ defines an extremal K\"ahler metric.  
 We consider the finite dimensional subspace that is transverse to the complexified orbit of $J_0$
\begin{eqnarray*}
H^1_G=\{\alpha\in\Omega^{0,1}(T^{1,0})^G\ ;\ P^*\alpha=0,\ \bar\partial\alpha=0\}.
\end{eqnarray*}
 Let $K\subseteq \mathcal{G}$ be the connected component of the identity of the stabilizer subgroup for the element $J_0\in \mathcal{J}^G$.  $K$ is a compact Lie group and acts complex linearly on the vector space $H^1_G$. The complexification $K^\C$ of the group $K$ also acts on $H^1_G$.

We can recall a result from \cite{sz10}.
\begin{proposition}\label{prop:slice}
There is a ball centered at the origin $B\subseteq H^1_G$ and a $K$-equivariant map $\Phi:B\to \mathcal{J}^G$ such that :
\begin{itemize}
\item $\Phi(0)=J_0$,
\item the complexified orbit of every integrable complex structure $J\in\mathcal{J}^G$ close to $J_0$ intersects the image of $\Phi$, 
\item $\Phi$ is $K^\C$-equivariant in the sense that if $x$ and $x^\prime$ are in the same $K^\C$-orbit and $\Phi(x)$ is integrable, then $\Phi(x)$ and $\Phi(x^\prime)$ lie in the same $\mathcal{G}^\C$-orbits in $\mathcal{J}^G$,
\item the moment map $\mu^G=S^G$ takes values in $\mathfrak{k}\subseteq Lie(\mathcal{G}) $ along the image of $\Phi$,
\item $\Phi^*\Omega$ is a symplectic form on $B$.
\end{itemize}
\end{proposition}
Then, since $\Phi$ is $K$-equivariant, $\mu=\Phi^*S^G$ defines a moment map on $B$ for the action of $K$. With this reduction of the problem of finding extremal metrics to a finite dimensional problem, one can more directly apply the ideas of geometric invariant theory. 

\begin{proposition}\label{prop:polyst}
\emph{\cite{sz10}} Let $x\in B$ be polystable for the action of $K^\C$ on $H^1_G$. Then there exists $x^\prime\in B$ in the $K^\C$-orbit of $x$ such that $\mu(x)=S^G(\Phi(x))=0$.
\end{proposition}
That is, we have an explicit algebraic criterion for when nearby complex orbits also contain extremal metrics.\\

We now turn to the construction of a $G$-invariant test configuration, as in Proposition \ref{prop:test}.

\begin{proposition}\label{prop:TCexistence}
Let $(X, L)$ be a polarised manifold that admits an extremal metric in the K\"ahler class $c_1(L)$. Let $\mathcal{L}\to\mathcal{X}\to\mathcal{B}$ be a $G$-invariant deformation of $(X,L)=(X_{t_0},L_{t_0})$ where $G$ is a compact connected Lie group that acts on the fibres by reduced automorphisms and such that $JV^{G^m}$, for some maximal subgroup $G^m$ that contains $G$, is infinitesimally generated by $G$. 

Then for every $t\in\mathcal{B}$ sufficiently close to $t_0$ there exists a $G$-invariant test configuration $\mathcal{L}_T\to \mathcal{X}_T\to \C$ with generic fibre isomorphic to $(X_t,L_t)$ such that the central fibre $(\cX_0,\cL_0)$ admits a $G$-invariant extremal metric in the K\"ahler class $c_1(\cL_0)$. The vector field $J_0V^{G^m_0}$ on the central fibre of $\cX_T$ also lies in $Lie(G)$. 
\end{proposition}

We first require a short lemma. Recall that $K$ lies in the quotient $ Z(G,H)/(G\cap Z(G,H))$. While $K$ does not act on $M$, we can lift its Lie algebra $\k$ to lie in the Lie algebra of $Z(G,H)$ and so that it gives an infinitesimal action on $M$:

\begin{lemma}
\label{lem:action}
There is a lift of $\k$ to the Lie algebra $\mathfrak{z}$ of $Z(G,H)$ whose image lies in the Lie algebra of hamiltonian Killing vector fields  for $g_0$. 
\end{lemma}

\begin{proof}
A priori, $\mathfrak{k}$ lies in $\mathfrak{z}/(\mathfrak{g}\cap\mathfrak{z})$, but this space is isomorphic to $(\mathfrak{g}\cap\mathfrak{z})^\perp\subseteq\mathfrak{z}$, where we consider the spaces to be sets of smooth $G$-invariant functions and we take the orthogonal with respect to the $L^2$-norm on functions, equipped with the Poisson bracket. 
\end{proof}

As a corollary, we deduce that each $\C^*$-subgroup of $K$, induces a $\C^*$-action on $X$ by biholomorphisms. Let $\tilde{K}\subseteq Z(G,H)$ be the preimage of $K$ by the projection $\pi: Z(G,H)\to Z(G,H)/(G\cap Z(G,H))$, and let $\tilde{K}^\C$ be its complexification. We note that $\tilde K$ acts on $(M,\omega,J_0)$ by hamiltonian isometries. It then follows (see \cite{Don01,kob}) that $\tilde K$ lifts to act on the bundle $L_0$.

\begin{corollary}
\label{cor:1ps}
Let $\rho:\C^*\to K^\C\subseteq (Z/(G\cap Z))^\C$ be a one-parameter-subgroup. Then there exists a one-parameter subgroup $\tilde\rho:\C^*\to {\tilde K}^\C$ such that
\begin{enumerate}
\item $\tilde\rho|_{S^1}:S^1\to \tilde K$,
\item $\pi\circ\tilde\rho=\rho$.
\end{enumerate}
\end{corollary}

\begin{proof}
We consider the map $\rho_*:S^1\to Z/(G\cap Z)$. On the level of Lie algebras, we can suppose that $\rho_*$ lifts to $\rho^\prime_*:Lie(S^1)\to \mathfrak{z}$. Suppose $Lie(S^1)$ is generated by $v$. The map $\rho^\prime_*$ can be integrated to a homomorphism 
\begin{eqnarray*}
\rho^\prime:\R&\to &Z,\\
t&\mapsto & \exp(t\rho^\prime_*(v))
\end{eqnarray*}
 that covers the homomorphism $\rho$. For $t_0>0$, if $\rho(t_0)=1\ (mod.\ G)$  then $\rho^\prime(t_0)=\gamma=\exp(w)$ for some $w\in \mathfrak{z}\cap\mathfrak{g}$. We can then define $\tilde\rho$ by $\tilde\rho(t)=\exp(t(\rho^\prime_*(v)-1/t_0w))$. Then, 
 \begin{eqnarray*}
 \tilde\rho(t_0)=\exp(t_0\rho^\prime_*(v))\cdot\exp(-w)=1
 \end{eqnarray*}
 since $w$ and $\rho^\prime_*(v)$ commute.
\end{proof}

\begin{proof}[ {\it Proof of Proposition \ref{prop:TCexistence}}] For $t\in\mathcal{B}$ sufficiently close to $t_0$, $J_t$ lies in the complexified orbit of $\Phi(x)$ for some $x\in B$. 
If $x\in B$ is polystable for the action of $K^\C$ on $H^1_G$, then by Proposition \ref{prop:polyst}, $\mu(\Phi(x^\prime))=0$ for some $x^\prime$ in the same $K^\C$-orbit of $x$, which is to say that $(M,J_t)$ admits an extremal metric in $c_1(L_t)$. We can take a trivial test configuration $\mathcal{X}_T=X\times\C$ with the desired properties. 

If $x$ is not polystable, then there exists a $1$-parameter subgroup of $K^\C$
\begin{eqnarray*}
\rho:\C^*\to K^\C
\end{eqnarray*}
such that $\lim_{\lambda\to 0}\rho(\lambda)\cdot x=x_0\in B$ is stable.
 Then, the map $\rho$ extends to a holomorphic map 
\begin{eqnarray*}
\rho(\cdot)\cdot x:\C\to H^1_G.
\end{eqnarray*}
We recall from the proof of Proposition \ref{prop:slice} in \cite{sz10} that $\Phi$ is obtained as a smooth deformation along  complexified orbits of $\Phi_1:B\to\mathcal{J}^G$, where $\Phi_1$ is holomorphic and $K^\C$-equivariant in the same sense as $\Phi$. That is, $\Phi(x)$ and $\Phi_1(x)$ always lie in the same complexified orbit. With $\rho$ we can then consider the holomorphic map 
\begin{eqnarray*}
F:\Delta &\to &\mathcal{J}^G\\
z &\mapsto & \Phi_1(\rho(z)\cdot x)
\end{eqnarray*}
where $\Delta\subseteq \C$ is a small disk of radius $\delta$ centred at the origin. Let $\mathcal{X}_\delta=M\times \Delta$, and equip $\mathcal{X}_\delta$ with the almost complex structure given by the usual structure on $\Delta$, and by $J_z=F(z)$ on the fibres $M\times \{z\}$. This structure is integrable, since the Nijenhuis tensor $N_J(X,Y)$ vanishes if $X$ and $Y$ are both tangent to one of the two factors, and if $X\in T_\Delta$ and $Y\in T_M$,
\begin{eqnarray*}
N_J(X,Y) &=& \frac{1}{4}\left((\mathcal{L}_{JX}J)Y-J(\mathcal{L}_XJ)Y\right)
\end{eqnarray*}
which vanishes since $F$ is holomorphic. In each fibre of the product $M\times \Delta$, the complex structures commute with the action of $G$ on $M$, so we can see that $\mathcal{X}_\delta$ admits a holomorphic action of $G$. \\

Let $\tilde\rho:\C^*\to\tilde K^\C$ be a $1$-parameter subgroup, lifted from a subgroup of $K^\C$, as in Corollary \ref{cor:1ps}, and suppose that $\tilde\rho(S^1)\subseteq \tilde K$. The subgroup then partially acts on
 $\mathcal{M}_\delta=M\times \Delta$ by 
\begin{eqnarray}\label{eqn:CstAction}
\lambda\cdot (x,z)=(\tilde{\rho}(\lambda)\cdot x,\lambda z)
\end{eqnarray}
where the expression holds for $\lambda\in\C^*$ and $z\in\Delta$ such that $\lambda z\in\Delta$.  From the equivariance of the map $\Phi_1$, if $z$ and $\lambda z$ lie in $\Delta$, then $J_{z}=\tilde{\rho}(\lambda)^*J_{\lambda z}$, and so the action of $\lambda\in\C^*$, where it is defined, is by holomorphic maps. This can be extended to a fibration over $\C$ that admits a $\C^*$-action as follows. Fix $z\in \Delta$ and consider the manifold $\mathcal{X}_T=M\times \C$, equipped with the complex structure on the fibre over $\lambda z$ given by $\tilde{\rho}(\lambda^{-1})^*J_z$. Then via (\ref{eqn:CstAction}), $\C^*$ acts on $\mathcal{X}_T$ by automorphisms, preserving the central fibre and inducing the action by scalar multiplication on $\C$.

Let $\overline{L}$ be the smooth complex line bundle on $M$ that underlines $L$ and $L_t$ and let $\nabla^t$ be a connection on $\overline{L}$ that determines the holomorphic structure on $L_t$ with respect to $J_t$. Assume that $\nabla^t$ is $S^1$-invariant (note, the $(0,1)$-part is not invariant). Again consider the product $\mathcal{X}_T=M\times \C$ and the projection $\pi:\mathcal{M}\to M$. Set $\mathcal{L}_T=\pi^*(\overline{L})$ as a line bundle on $\mathcal{X}_T$ and with connection $\nabla=\pi^*(\nabla^t)$.
Then, $F^{0,2}_\nabla=0$ on $\mathcal{X}_T$ and $\nabla $ defines a holomorphic structure on $\mathcal{L}_T$. $S^1$ acts holomorphically on $\mathcal{L}_T$ and this extends to a $\C^*$-action that covers the $\C^*$-action on $\mathcal{X}_T$. 

We thus obtain a test-configuration with generic fibre isomorphic to the polarized manifold $(X_t,L_t)$, and for which the central fibre $(\cX_0,\cL_0)$ admits an extremal metric in the K\"ahler class $c_1(L_0)$. The final statement of Proposition \ref{prop:TCexistence} follows from Proposition \ref{prop:polyst}. As  $x_0$ is stable, $S^G(\Phi(x_0))=0$ and the scalar curvature of the extremal metric on $\cX_0$ belongs to the space of Killing potentials of $\g$.
\end{proof}


\section{Lower bounds and deformations}
\label{sec:prooftheo}

In this section we turn to the argument of Tosatti for the boundedness of the Mabuchi energy under small deformations and consider the relative Mabuchi energy. In the case at hand we assume that the deformation preserves a group of automorphisms. We assume that $J$ times the extremal vector fields take values in $\mathfrak{g}$, for the central fibre of the deformation, and for some nearby fibre. 

That is, let $(X^\prime, L^\prime)$ be a polarized complex manifold, that admits an extremal metric with K\"ahler class $c_1(L^\prime)$. Let $G$ be a compact connected group of automorphisms of $(X^\prime,L^\prime)$ such that $JV^{G^m}$ lies in $\mathfrak{g}\subseteq \mathfrak{aut}(X^\prime,L^\prime)$ for some maximal compact subgroup $G^m$ of the reduced automorphism group. Let $\mathcal{L}\to\mathcal{X}\to\mathcal{B}$ be a $G$-invariant polarized deformation of $(X^\prime, L^\prime)$, with $(X,L)=(X_t,L_t)$ a fibre sufficiently close to the central fibre. 
For some maximal compact connected subgroup $G^m_t$ of $Aut(X,L)$, denote by $V^{G^m_t}$ the extremal vector field on $(X,L)$.






The manifold $(X^\prime, L^\prime)$ admits a $G$-invariant extremal metric $\omega^\prime\in c_1(L^\prime)$, so from \cite{ct,st} the modified K-energy is bounded below on the set of $G$-invariant K\"ahler potentials with respect to $\omega^\prime$ on $X^\prime$. We show the following theorem. 

\begin{theorem}
Let $(X,L)=(X_t,L_t)$ be a sufficiently close fibre of a $G$-invariant deformation of $(X^\prime,L^\prime)$. For some maximal compact subgroup $G^m_t$ of $Aut(X_t,L_t)$ that contains $G$, suppose that $J_tV^{G^m_t}$ lies in $\mathfrak{g}$ on $X_t$. Then for any $G$-invariant K\"ahler form $\omega\in c_1(L_t)$, the modified K-energy $E^{G^m_t}$ is bounded below on $G^m_t$-invariant K\"ahler potentials.
\end{theorem}




 We recall Proposition \ref{prop:TCexistence} and suppose that $\mathcal{L}_T\to \mathcal{X}_T\stackrel{\pi}{\to} \C$ is a $G$-invariant test configuration with generic fibre $(X,L)$, and where the central fibre $(X_0,L_0)$ admits an extremal metric in $c_1(L_0)$ and such that the extremal vector field is contained in $\mathfrak{g}$. From \cite{st}, the modified K-energy, with respect to $G$, of $(X_0,L_0)$ is bounded below.  Let $\rho$ denote the $\C^*$-action on $(\mathcal{X}_T,\mathcal{L}_T)$. For fixed $\lambda\in\C^*$ this will be denoted alternately $\rho_\lambda$ or $\rho(\cdot,\lambda)$. In particular, $\rho$ generates a holomorphic vector field on the central fibre $X_0$ that commutes with vector fields in $\mathfrak{g}$.
 

By the theorem of Ehresmann, the fibration $\mathcal{X}_T\to\C$ is differentiably trivial. That is, there exists a diffeomorphism 
\begin{eqnarray*}
F:M\times \C& \to &\mathcal{X}_T\\
\text{such that }\ \ \ \pi(F(x,z))&=& z.
\end{eqnarray*}
The action of $G$ on $\mathcal{X}_T$ gives a family of actions on $M$. 
As in the previous section, by \cite{ps} we can suppose that there is a fixed action of $G$ on $M$ such that, for $\sigma\in G$,  
\begin{eqnarray*}
F(\sigma\cdot x,\lambda)=\sigma\cdot F(x,\lambda).
\end{eqnarray*}

Let $\cX_0=\pi^{-1}(0)$ be the central fibre of $\mathcal{X}_T$ equipped with complex structure $J_0$. Assume that the embedding $F:X_0\to \mathcal{X}_T$ is a biholomorphism to its image, for $X_0=(M,J_0)$.

Let $\cX_1=\pi^{-1}(1)$ be a generic fibre of $\mathcal{X}_T$. Then, we can trivialize $\mathcal{X}_T\setminus \pi^{-1}(0)$ over $\C^*$ using the $\C^*$-action on $\mathcal{X}_T$ that we constructed in the previous section. That is,
\begin{eqnarray*}
\rho:X_1\times \C^* &\to &\mathcal{X}\setminus \pi^{-1}(0)\\
(x,\lambda) &\mapsto & \rho(x,\lambda).
\end{eqnarray*}
This trivialisation is biholomorphic and commutes with the action of $G$ in the sense that
\begin{eqnarray*}
\rho(\sigma\cdot x,\lambda)=\sigma\cdot\rho(x,\lambda).
\end{eqnarray*}
The two trivializations $F$ and $\rho$ can be combined to define a $1$-parameter family of diffeomorphisms $f_\lambda:M\to M$ such that $F(x,\lambda)=\rho_\lambda(f_\lambda(x))$ for all $x\in M$. 

As in \cite{to}, we can suppose that we have an $S^1$-invariant K\"ahler form $\Omega$ on $\mathcal{X}_T$, where $S^1\subseteq \C^*$ is the compact subgroup arising from the action of $\rho$.  Suppose also that $\Omega|_{\cX_t}$ lies in $c_1(\cL_t)$ and that the action of $S^1$ is hamiltonian. That is, there is a smooth function $H:\mathcal{X}_T\to \R$ such that 
\begin{eqnarray*}
i_W\Omega=dH
\end{eqnarray*}
where $W$ generates the $S^1$-action on $\mathcal{X}_T$. We also assume that $\Omega$ is $G$-invariant and that the induced metric on the central fibre $X_0=\pi^{-1}(0)$ satisfies $S^G(\Omega|_{\cX_0})=0$. That is, it is extremal. 


 Since $W$ generates the $S^1$-action, the vector field $-JW$ generates the real flow $t\mapsto   \rho_{e^{-t}}(x)$. Since $\rho_{e^{-t}}$ is a holomorphic map of $\mathcal{X}_T$, $\omega_t=\rho^*_{e^{-t}}\Omega$ defines a family of K\"ahler forms that lie in the same cohomology class. It then follows that 
\begin{eqnarray*}
\frac{d}{dt}\omega_t= \rho^*_{e^{-t}}\mathcal{L}_{-JW}\Omega =i\partial \bar\partial \rho^*_{e^{-t}}H.
\end{eqnarray*}
On the other hand, the forms are cohomologous so there exists a family of potentials $\varphi_t$ such that 
\begin{eqnarray*}
\frac{d}{dt}\omega_t=i\partial\bar\partial\dot\varphi_t
\end{eqnarray*}
so, modulo constants, $\dot\varphi_t=\rho^*_{e^{-t}}H$. 



Let $g_\Omega$ denote the metric on $\mathcal{X}_T$ associated to the form $\Omega$ and consider the family of metrics $g_t=\rho^*_{e^{-t}}g_\Omega$ on $X_1$. Then, since $F_\lambda=\rho_\lambda\circ f_\lambda$ is defined smoothly across $\lambda=0$, the metrics $g_t$ and forms $\omega_t$ satisfy the inequalities
\begin{eqnarray*}
\|f^*_{e^{-t}}\omega_t-F^*_0\Omega\|_{C^k}&<&C_ke^{-t}\\
\|f^*_{e^{-t}}g_t-F^*_0g_\Omega\|_{C^k} &<& C_k e^{-t}
\end{eqnarray*}
and the curve of potentials satisfies
\begin{eqnarray*}
|f^*_{e^{-t}}\dot\varphi_t-F_0^*H|<Ce^{-t}.
\end{eqnarray*}


We recall the definition from Section \ref{sec:modified} of the modified Calabi energy, relative to the group $G$, of a K\"ahler potential $\varphi$,
\begin{eqnarray*}
Ca^G(\varphi)=\int S^G(\omega_\varphi)^2d\mu_\varphi
\end{eqnarray*}
where $S^G(\omega_\varphi)$ is the reduced scalar curvature of the metric $\omega_\varphi$. Given the group $G$, which acts by hamiltonian diffeomorphisms with respect to a fixed symplectic form, the reduced scalar curvature is purely riemannian. That is, if we specify a finite dimensional space of functions to project away from, the reduced scalar curvature and volume form depend only on the metric. On the manifold $(\cX_1,\cL_1)$ then,
\begin{eqnarray*}
Ca^G(\varphi_t)&=&\frac{1}{n!}\int S^G(\omega_t)^2\omega^n\\
&=&\frac{1}{n!}\int S^G(f^*_{e^{-t}}\omega_t)^2(f^*_{e^{-t}}\omega_t)^n
\end{eqnarray*}
which converges exponentially fast to 
\begin{eqnarray*}
\frac{1}{n!}\int S^G(F^*_0\Omega)^2(F^*_0\Omega)^n.
\end{eqnarray*}
The metric on the central fibre is extremal and the extremal vector field is contained in $\mathfrak{g}$, so this value is equal to zero. Similarly, the derivative of the modified K-energy satisfies
\begin{eqnarray*}
E^G(\phi)&=& -\int_0^1\int_X \dot{\phi}_t S^G(\omega_{\phi_t}) d\mu_{\phi_t}\\
\frac{d}{dt}E^G(\varphi_t)&=& -\frac{1}{n!}\int_X\dot{\varphi}_t S^G(\omega_t)\omega^n_t\\
&=& -\frac{1}{n!}\int_X(f^*_{e^{-t}}\dot{\varphi}_t)S^G(f^*_{e^{-t}}\omega_t)(f^*_{e^{-t}}\omega_t)^n. 
\end{eqnarray*}
This converges exponentially fast to the value
\begin{eqnarray*}
-\frac{1}{n!}\int (F^*_0H)S^G(F^*_0\Omega)(F^*_0\Omega)^n
\end{eqnarray*}
which can be seen to equal (minus) the relative Futaki invariant (see Defn. \ref{defn:FutChar}) on the central fibre $X_0$ evaluated on the real holomorphic vector field $W$. This vanishes since the central fibre is supposed to admit an extremal metric.

Let $\omega=\Omega|_{\cX_1}$ be a K\"ahler metric contained in $c_1(\cL_1)$. Let $\varphi$ be any $G$-invariant K\"ahler potential, relative to $\omega$. For a fixed $t_0$, join $\varphi$ to $\varphi_{t_0}$ by a piecewise smooth curve. We can concatenate this with the curve $\varphi_t$ that is given above, starting at $\varphi_{t_0}$. 

We can then apply the inequality of Proposition~\ref{prop:convex} to see that for any K\"ahler potential $\varphi$, relative to $\omega$, 
\begin{eqnarray*}
E^G(\varphi)\geq E^G(\varphi_t)-\sqrt{Ca^G(\varphi_t)}\int_0^t\sqrt{\int_{X_1}\dot\varphi_s\omega_{\varphi_s}^n}ds
\end{eqnarray*}
The derivative of the first term on the right converges exponentially to zero, so $E^G(\varphi_t)$ is bounded below as $t$ increases. The other term can also be controlled, since the modified Calabi invariant converges exponentially to zero while the integral grows at most linearly in $t$. We can conclude that there exists $C\in \R$ such that 
\begin{eqnarray}\label{eqn:ETbound}
E^G(\varphi)\geq -C
\end{eqnarray}
for every $G$-invariant K\"ahler potential $\varphi$ in $c_1(L_1)$. Since the extremal vector field $JV^{G^m}$ takes values in $J\mathfrak{g}$ this implies that $E^{G^m}$ is uniformly bounded below.





\section{Application}
\label{sec:ap}
Let $X$ be the blow-up of $\C\P^1\times\C\P^1$ at its four fixed points under the torus action
$$
\begin{array}{ccc}
\mathbb{T}^2\times \C\P^1\times\C\P^1 & \rightarrow & \C\P^1\times\C\P^1 \\
\left( (\theta,\theta'),([x_1,y_1],[x_2,y_2]) \right) & \mapsto & ([e^{i\theta}x_1,y_1],[e^{i\theta'}x_2,y_2])
\end{array}
$$
The deformation space of this complex manifold has been studied in \cite{rt}, following works of Ilten and Vollmert \cite{iv}. We can endow $X$ with an extremal metric of non-constant scalar curvature and prescribed extremal vector field periodic action. Start with a product constant scalar curvature K\"ahler metric $\om$ on $\C\P^1\times\C\P^1$. Assume that the restriction of $\om$ on each factor of $\C\P^1\times \C\P^1$ has same volume. From Arezzo-Pacard-Singer theorem \cite{aps}, for each $(a_1,a_2,b_1,b_2)$ positive numbers, $X$ admits an extremal metric $\om_\ep$ in the class
$$
[\pi^*\om]-\ep^2(a_1 PD(E_{0,0}) + a_2 PD(E_{\infty,0})+ b_1 PD(E_{0,\infty}) + b_2 PD(E_{\infty,\infty}))
$$
for $\ep$ positive small enough, and where $\pi$ denotes the blow-down map, $PD(E)$ is the Poincar\'e dual of $E$ and $E_{i,j}$ is the exceptional divisor associated to the blow-up of the point $(i,j)\in\C\P^1\times\C\P^1$. 
To prescribe the extremal vector field, we consider the class
$$
[\om_\ep]=[\pi^*\om]-\ep^2\left(a PD(E_{0,0}) + a PD(E_{\infty,0})+ b PD(E_{0,\infty}) + b PD(E_{\infty,\infty})\right)
$$
for $\ep$ positive small enough and $a\neq b$. 
The associated polytope is represented Figure 1.
Note that up to scaling, we can suppose that the class $[\om_\ep]$ is integral and represent a polarization $L$ of $X$.
\begin{figure}[htbp]
\psset{unit=0.85cm}
\begin{pspicture}(0,-6)(12,0)
\psframe[linecolor=white](0.5,-4.5)(3.5,-1.5)

\uput*[270](7,-3){$x$ axis}
\uput*[270](5.5,-1){$y$ axis}

\psline[linewidth=0.5pt, linestyle=dotted]{<->}(6,0)(6,-6)
\psline{-}(3.5,-5)(3.5,-1.5)
\psline{-}(8.5,-5)(8.5,-1.5)
\psline[linewidth=0.5pt, linestyle=dotted]{<->}(3,-3)(9,-3)
\psline{-}(4,-5.5)(8,-5.5)
\psline{-}(4.5,-0.5)(7.5,-0.5)

\psline{-}(3.5,-5)(4,-5.5)
\psline{-}(8,-5.5)(8.5,-5)
\psline{-}(8.5,-1.5)(7.5,-0.5)
\psline{-}(4.5,-0.5)(3.5,-1.5)

\uput*[270](3,-5.3){$E_{0,0}$}
\uput*[270](3.4,-0.1){$E_{0,\infty}$}
\uput*[270](9,-5.3){$E_{\infty,0}$}
\uput*[270](8.6,-0.1){$E_{\infty,\infty}$}
\end{pspicture}
\caption{Polytope associated to $(X, \om_\ep)$.}
\end{figure}
Following \cite{Don02} (see also \cite{Le}), we can compute the extremal vector field associated to this extremal metric with respect to the maximal compact group $\mathbb{T}^2\subset \Aut(X)$. The extremal vector field is invariant with respect to the isometry group of $\om_\ep$. By the axial symmetry of the polytope, the potential of the extremal vector field is an affine function on the polytope that only depends on the $y$ coordinate. As $a$ is chosen different from $b$, the Futaki invariant of $[\om_\ep]$ is different from zero and the extremal vector field does not vanish. Let $\mathbb{T}_f\subset\Aut(X)$ be the lift of the circle subgroup of $\Aut(\C\P^1\times \C\P^1)$ defined by
$$
\begin{array}{ccc}
S^1\times \C\P^1\times\C\P^1 & \rightarrow & \C\P^1\times\C\P^1 \\
(\theta ,([x_1,y_1],[x_2,y_2]))& \mapsto & ([x_1,y_1],[e^{i\theta}x_2,y_2]).
\end{array}
$$
Then by construction the extremal vector field of $\om_\ep$ generates the action of $\mathbb{T}_f$ on $X$.


Now, from the study of the example 4.2. in the article \cite{rt}, the space of infinitesimal complex deformations of $X$ that preserve the $\T_f$-action $H^1(X,\Theta_X)^{\mathbb{T}_f}$ is isomorphic to $\C^2$. 
The automorphism group of $X$ admits the splitting
$$
\Aut(X)=\mathbb{T}_f^\C\times \mathbb{T}_a^\C
$$
where $\mathbb{T}_a^\C\simeq \C^*.$
Then $\mathbb{T}_a^\C$ acts on $H^1(X,\Theta_X)^{\mathbb{T}_f}$:
$$
\begin{array}{ccc}
\mathbb{T}_a^\C \times H^1(X,\Theta_X)^{\mathbb{T}_f} &\rightarrow & H^1(X,\Theta_X)^{\mathbb{T}_f} \\
 (\lambda, (x,y)) & \mapsto & (\lambda^{-1} x,\lambda y).
\end{array}
$$
By Theorem 3.3.1 in \cite{rt}, the closed orbits under this action induce deformations of $X$ that carry extremal metrics. Those with non-closed orbits, called unstable, induce deformations of $X$ that carry no extremal metric, while the extremal vector field action is preserved. Indeed, using Proposition \ref{prop:TCexistence}, if $(X',L')$ is a small deformation of $(X,L)$ associated to an unstable infinitesimal deformation $\xi \in H^1(X,\Theta_X)^{\mathbb{T}_f}$, we can build a test-configuration for $(X',L')$ which is compatible with $\mathbb{T}_f$. By construction, this test configuration is not trivial, and the associated relative Donaldson-Futaki invariant vanishes as its central fiber is extremal. Then $(X',L')$ is not K-stable relative to $\mathbb{T}_f$. As $\mathbb{T}_f$ is a maximal torus in $\Aut(X')$, by the result of Stoppa and Sz\'ekelyhidi \cite{stsz}, $X'$ carries no extremal metric in $c_1(L')$. However, by Theorem~\ref{theo:defo}, this polarized manifold has bounded modified K-energy.


\begin{thebibliography}{BP}

\bibitem{acgt}
V. Apostolov, D.M.J. Calderbank, P. Gauduchon and C.W. Toennesen-Friedman, {\it Extremal
K\"ahler metrics on projective bundles over a curve}, Adv. Math. {\bf 227} (2011), 2385-2424.

\bibitem{aps} 
C.Arezzo, F.Pacard and M.Singer,
{\it Extremal Metrics on blow ups}, Duke Math. J. Volume \textbf{157}, Number 1 (2011), 1-51. 

\bibitem{br}
T.Broennle, {\it Deformation constructions of extremal metrics}, Phd Thesis.

\bibitem{c1}
E. Calabi, {\it Extremal K\"ahler metrics}, Seminars on Differential 
Geometry (S. T. Yau Ed.), Annals of Mathematics Studies, Princeton 
University Press, 1982, pp. 259--290.

\bibitem{c2}
E. Calabi, {\it Extremal K\"ahler Metrics II}, 
Differential Geometry and Complex Analysis (eds. Chavel \& Farkas), 
Springer-Verlag, 1985, pp. 95--114.

\bibitem{cat}
D. Catlin, {\it The Bergman kernel and a theorem of Tian}, Analysis and geometry in several complex variables (Katata, 1997), 
Trends Math., 1--23. Birkh\"auser Boston, Boston, MA, 1999.

\bibitem{ch}
X.X. Chen, {\it Space of K\"ahler metrics. III. On the lower bound of the Calabi energy and geodesic distance}.  Invent. Math.  175  (2009), (3), pp. 453-503.

\bibitem{ch2}
X.X. Chen, {\it Space of K\"ahler metrics (IV) On the lower bound of the K-energy}, preprint arXiv:0809.4081.

\bibitem{cs} X.X. Chen and S.Sun {\it Space of K\"ahler metrics (V)- K\"ahler quantization}, ''Metric and Differential Geometry:
The Jeff Cheeger Anniversary Volume'', 19--41, Progress in Mathematics 297, Birkh\"auser, 2012.


\bibitem{ct}
X.X. Chen, G.Tian,  {\it Geometry of K\"ahler metrics and foliations by holomorphic discs} ,  Publ. Math. Inst. Hautes \'Etudes Sci. {\bf 107} (2008), 1-107.


\bibitem{Don97} S. K. Donaldson {\it Remarks on gauge theory, complex geometry and four-manifold topology.}, 
In Atiyah and Iagolnitzer, editors, {\it Fields Medallists' Lectures}, 384-403. World Scientific, (1997).

\bibitem{Don01} S. K. Donaldson {\it Scalar curvature and projective embeddings. I.}, I. J. Differential Geom. {\bf 59}(2001), no.3,  479-522.

\bibitem{Don02}
S. K. Donaldson, {\it Scalar curvature and stability of toric varieties}, J. Differential Geom. {\bf 62}(2002), 289-349.

\bibitem{Don05} S. K. Donaldson {\it Scalar curvature and projective embeddings. II.}, 
Q.J.Math. 56, no.3 (2005) 345-356.

\bibitem{fuj} A. Fujiki {\it Moduli space of polarized algebraic manifolds and K\"ahler metrics}, Sugaku Exp., Vol. 5, No. 2. (1992), 

\bibitem{fut} 
A. Futaki,  {\it An obstruction to the existence of 
K\"ahler-Einstein metrics}, Invent. Math., {\bf 73} (1983), pp. 437-443.

\bibitem{futbook}
A. Futaki, {\it K\"ahler-Einstein metrics and Integral Invariants}, Springer-LNM 1314,
Springer-Verlag (1988).

\bibitem{fm}
A. Futaki \& T. Mabuchi, {\it Bilinear forms and
extremal K\"ahler vector fields associated with K\"ahler classes},
Math. Annalen, 301 (1995), pp. 199--210.

\bibitem{gbook}
P. Gauduchon, {\it Calabi's extremal metrics: An elementary introduction}, 
book in preparation (2011).

\bibitem{gu}
D.Guan, {\it On modified Mabuchi functional and Mabuchi moduli space of K\"ahler metrics on toric bundles}, Mat. res. Letters {\bf 6}, 547-555 (1999).

\bibitem{iv}
N.O.Ilten \& R.Vollmert, {\it Deformations of rational T-varieties} J. Algebraic Geometry 21 (2012) pp. 473--493

\bibitem{kob}
S.Kobayashi, {\it Transformation groups in differential geometry}, Ergebnisse der Mathematik und ihrer Grenzgebiete, {\bf 70}, Springer-Verlag, New York-Heidelberg, 1972.

\bibitem{ku}
M.Kuranishi, {\it New proof for the existence of locally complete families of complex structures.} In 
{\it Proc. Conf. Complex Analysis (Minneapolis 1964)}, 142-154,. Springer, Berlin, 1965.


\bibitem{Le}
E. Legendre {\it Toric geometry of convex quadrilaterals} J. Symplectic Geom. Volume 9, Number 3 (2011), 343--385. 

\bibitem{li}
C.Li , {\it Constant Scalar Curvature K\"ahler metrics Obtains the Minimum of K-energy},
Int. Math. Res. Not., Vol. 2011, No. 9, pp. 2161-2175.

\bibitem{lic}
A. Lichnerowicz, {\it G\'eom\'etrie des groupes de transformation.}, Travaux et recherches math\'ematiques 3, Dunod (1958).

\bibitem{ma}
T. Mabuchi, {\it $K$-energy maps integrating Futaki invariants}, Tohoku Math. J. vol. {\bf 38} (no.4), 1986, 575--593.

\bibitem{ma04}
T. Mabuchi, {\it Uniqueness of extremal K\"ahler metrics for an integral K\"ahler class }, Internat. J. Math. {\bf 15}, (2004), 531-546.

\bibitem{ps} 
R. Palais \& T.E. Stewart, {\it Deformations of compact differentiable 
transformation groups}, Amer. J. Math. 82 (1960), pp. 935-937.

\bibitem{rst}
Y.Rollin, S.Simanca, C.Tipler , {\it Stability of extremal metrics under complex deformations}, to appear in Math. Zeit., available at ArXiv 1107.0456.

\bibitem{rt}
Y.Rollin \& C.Tipler, {\it Deformations of Extremal Toric manifolds}, to appear in J. Geom. An. preprint arXiv:1201.4137.

\bibitem{ruan}
W.-D. Ruan, {\it Canonical coordinates and Bergman metrics}, Comm. Anal. Geom. {\bf 6} (1998), no. 3,589-631.

\bibitem{si}
S. R. Simanca, {\it A K-energy characterization of extremal K\"ahler metrics}, Proc. Amer. Math. Soc. {\bf 128} (2000), 1531-1535.

\bibitem{st}
Y.Sano \& C.Tipler, {\it  Extremal metrics and lower bound of the modified K-energy}, preprint arXiv:1211.5585

\bibitem{stsz} 
J. Stoppa \& G. Sz\'ekelyhidi, {\it Relative $K$-stability of
 extremal metrics}, J. Eur. Math. Soc. 13 (2011) n. 4, 899--909.

\bibitem{sz}
G.Sz\'ekelyhidi, {\it Extremal metrics and K-stability}, Bull. London Math. Soc. {\bf 39} (2007), 76-84.

\bibitem{sz10}
G. Sz{\'e}kelyhidi, {\it The {K}\"ahler-{R}icci flow and {$K$}-polystability},
   Amer. J. Math.,
 132 (2010), no {4}, pp. 1077-1090.
 
 \bibitem{tian90}
 G.Tian, {\it On a set of polarized K\"ahler metrics on algebraic manifolds,} J. Differential Geom. {\bf 32} (1990), 99-130.
 
\bibitem{tian}
G. Tian, {\it K\"ahler-Einstein metrics with positive scalar curvature}, Invent. Math. {\bf 137} (1997), 1-37.

\bibitem{to}
V.Tosatti, {\it The K-energy on small deformations of constant scalar curvature K\"ahler manifolds},
Advances in Geometric Analysis, 139-150, Advanced Lectures in Math. 21, International Press, 2012. 

\bibitem{yau}
S.-T. Yau, {\it Open problems in geometry}, Proc. Symposia Pure Math. {\bf 54} (1993), 1-28.

\bibitem{zel}
S. Zelditch, {\it Szego kernels and a theorem of Tian.} Internat. Math. Res. Notices, (1998) 317--331.
\end{thebibliography}
\end{document}